\documentclass[11pt]{amsart}
\usepackage[utf8]{inputenc}
\usepackage[pdftex]{graphicx}
\usepackage{amssymb}
\usepackage{amsmath}
\usepackage{amsfonts}
\usepackage{amsthm}
\usepackage{bm}
\usepackage{mathrsfs}
\usepackage{appendix, comment}
\usepackage{enumitem}
\usepackage[margin = 1.1in]{geometry}
\usepackage{stmaryrd}
\usepackage{mathtools}
\mathtoolsset{showonlyrefs}
\usepackage[normalem]{ulem} 

\usepackage{xcolor}
\definecolor{antiquefuchsia}{rgb}{0.57, 0.36, 0.51}
\definecolor{azure}{rgb}{0.0, 0.5, 1.0}
\definecolor{indigo}{rgb}{0.0, 0.25, 0.42}

\definecolor{linkcolour}{rgb}{0,0,0.7}
\definecolor{urlcolour}{rgb}{0,0,0.4}
\definecolor{citecolour}{rgb}{0.4,0.1,0}

\usepackage[backref=page, colorlinks = true, linkcolor = linkcolour, citecolor = citecolour]{hyperref}

\renewcommand*{\backref}[1]{}
\renewcommand*{\backrefalt}[4]{%
    \ifcase #1 (Not cited.)%
    \or        (Cited on page~#2.)%
    \else      (Cited on pages~#2.)%
    \fi}

\makeatletter
\def\th@plain{%
	\thm@notefont{}
	\itshape 
}
\def\th@definition{%
	\thm@notefont{}
	\normalfont 
}
\makeatother

\numberwithin{equation}{section}

\newtheorem{theorem}{Theorem}[section]
\newtheorem{lemma}[theorem]{Lemma}
\newtheorem{proposition}[theorem]{Proposition}

\theoremstyle{definition}
\newtheorem{definition}[theorem]{Definition}

\theoremstyle{remark}
\newtheorem{remark}[theorem]{Remark}

\newcommand{\N}{\mathbb{N}}

\newcommand{\dom}{\operatorname{dom}}

\newcommand{\mres}{\mathbin{\vrule height 1.6ex depth 0pt width
0.13ex\vrule height 0.13ex depth 0pt width 1.3ex}}

\DeclareMathOperator{\spt}{spt}

\DeclareMathOperator{\dist}{dist}

\DeclareMathOperator{\graph}{graph}

\newcommand{\uno}{\bm{1}}

\newcommand{\Hcal}{\mathcal{H}}

\newcommand{\Kcal}{\mathcal{K}}

\newcommand{\vbf}{\mathbf{v}}

\newcommand{\Rbb}{\mathbb{R}}
\newcommand{\Sbb}{\mathbb{S}}

\newcommand{\Tmax}{T_{\mathrm{max}}}

\newcommand{\oOm}{\overline{\Omega}}

\title[On the Gage--Hamilton--Grayson Theorem for free boundary curves]{On the Gage--Hamilton--Grayson Theorem for free boundary curves} 

\author[Yueheng Bao]{Yueheng Bao}
\address{University of Toronto, Department of Mathematics, 40 St George Street, Toronto, ON M5S 2E4, Canada}
\email{yueheng.bao@mail.utoronto.ca}

\author[Lorenzo Sarnataro]{Lorenzo Sarnataro}
\address{University of Toronto, Department of Mathematics, 40 St George Street, Toronto, ON M5S 2E4, Canada}
\email{lorenzo.sarnataro@utoronto.ca}

\begin{document}
\begin{abstract}
We provide an alternative and independent proof of the analogue of the Gage--Hamilton--Grayson Theorem for embedded free boundary curve shortening flows in a convex planar domain established by Langford--Zhu \cite{lz23}. Our proof is based on blowup arguments introduced by B. White in his work on mean-convex mean curvature flow. 
\end{abstract}
\maketitle
\thispagestyle{empty}
\vspace{-0.5cm}
\section{Introduction}\label{sec:introduction}
Curve shortening flow evolves a planar curve with normal velocity equal to its curvature vector. The Gage--Hamilton--Grayson theorem asserts that every embedded closed planar curve becomes convex and shrinks to a round point in finite time \cite{gh86,grayson87}.

Free boundary curve shortening flow evolves a properly embedded arc in a smooth bounded convex planar domain $\Omega$, with its endpoints constrained to lie on $\partial\Omega$ and the arc meeting $\partial\Omega$ orthogonally throughout the evolution. Langford and Zhu \cite{lz23} proved the free boundary analogue of the Gage--Hamilton--Grayson theorem: the flow either converges in infinite time to a single free boundary chord (namely, a line segment joining two points of $\partial \Omega$ and meeting $\partial \Omega$ orthogonally) or shrinks in finite time to a boundary point and becomes asymptotically semicircular. A key ingredient of their proof is a distance comparison principle, which is obtained by a long and sophisticated computation.

We give an alternative, less computationally heavy proof, based on the local blow-up argument in White's notes \cite[Sections~6--7]{White_lectures}, which is inspired by White's groundbreaking work on mean-convex mean curvature flow \cite{whi-size,whi-nature}. It is worth pointing out that the techniques in \cite{whi-size,whi-nature} were adapted to free boundary mean-convex flows in \cite{FB-mean-convex}. Our application is limited to free boundary flows of curves in convex domains in $\mathbb{R}^2$ but does not require the mean-convexity assumption. Our main result is the following.
\begin{theorem}\label{thm:main}
Let $\Omega\subset\Rbb^2$ be a nonempty bounded open convex domain with smooth boundary $S=\partial\Omega$. Let $\Gamma_0\subset\oOm$ be a properly embedded smooth arc meeting $S$ orthogonally, and let $\Gamma=\{\Gamma_t\}_{t\in[0,\Tmax)}$ be its maximal smooth free boundary curve shortening flow. Exactly one of the following holds:
\begin{enumerate}[label=(\roman*)]
\item $\Tmax=\infty$, and $\Gamma_t$ converges smoothly to a single chord meeting $S$ orthogonally at both endpoints;
\item $\Tmax<\infty$, and there is $p_0\in S$ such that
\begin{equation}
 \frac{\Gamma_t-p_0}{\sqrt{2(\Tmax-t)}}\to\Sbb^1_+
 \quad\text{smoothly as }t\uparrow\Tmax,
\end{equation}
where $\Sbb^1_+$ is the unit semicircle centred at the origin in $T_{p_0}\oOm$.
\end{enumerate}
In particular, in case \textup{(ii)}, the whole curve shrinks to $p_0$ and
 $\sqrt{2(\Tmax-t)}\,\|\kappa(\cdot,t)\|_{L^\infty(\Gamma_t)}\to 1$,
where $\kappa(\cdot,t)$ denotes the curvature of $\Gamma_t$.
\end{theorem}

\subsection{Outline}
The proof of Theorem \ref{thm:main} is divided into two cases, depending on the long-time behaviour of the flow. Whereas curve shortening flow of closed embedded planar curves always goes extinct in finite time, \emph{free boundary} curve shortening flow of an embedded arc in $\overline{\Omega}\subset \mathbb{R}^2$ can exist for all time and converge to a free boundary chord, as in case (i) of Theorem \ref{thm:main}.

We devote Section \ref{sec:immortal} to this case. We first prove that along any sequence of times going to infinity the flow converges subsequentially to a free boundary chord, and then that this chord is independent of the subsequence.\\
The existence of subsequential limits follows from the observation that the space-time $L^2$-norm of the geodesic curvature of a curve shortening flow is bounded on any compact set, which, by the Sobolev embedding theorem, implies local precompactness of the curve parametrisations in $C^{1, \alpha}$ (for $\alpha\in(0, 1/2)$) at almost every time. This use of the one-dimensional Sobolev embedding is a distinguishing feature of curve shortening flow (as opposed to higher-dimensional mean curvature flow), appears prominently in White's proof of Grayson's theorem in \cite{White_lectures}, and is used extensively throughout this paper. Uniqueness of the limiting chord is more subtle, and requires the new estimate \eqref{eq:main-estimate-uniqueness}, which allows us to control the distance from any fixed subsequential limit chord by the difference in length. As the limiting length is unique by the length-decreasing property of the flow, so must be the limiting chord.

The rest of the paper is then devoted to the analysis of case (ii), namely of the behaviour of the flow at the first singular time. By using the self-shrinking property of tangent flows to free boundary curve shortening flows (which is a consequence of Edelen's monotonicity formula \cite[Theorem 5.1]{edelen18}), the aforementioned Sobolev embedding, and Abresch--Langer's classification of smooth immersed self-shrinking curves in the plane \cite{al86}, we classify the possible tangent flows in Proposition \ref{prop:tangent-flow-classification} up to multiplicity. The shrinking semicircle already has multiplicity one; the remaining models are static at negative times.

The next step is to rule out higher multiplicity in these remaining models and to exclude the boundary tangent line altogether. In analogy with \cite{whi-size,whi-nature}, the key is to obtain a \emph{sheeting} result (cf. Proposition \ref{prop:sheeting}), namely, to show that, if a tangent flow appears with higher multiplicity, then the parabolically rescaled flows can be represented locally as graphs over the limit, with each sheet converging to the limit in $C^1$. We first prove graphicality at \emph{good times} (cf. Remark \ref{rmk:good-times}), where the Sobolev embedding already provides $C^{1, \alpha}$ precompactness of the parametrisations, and then extend it between good times by using a variation of White's intersection argument in \cite[Theorem 7.2, Lemma 7.4]{White_lectures}, which exploits the monotonicity properties of the number of zeros of solutions to parabolic equations in one space variable (see Lemma \ref{lem:zero}). Roughly speaking, we control graphicality by counting intersections with transverse segments (which can be seen as stationary solutions of curve shortening flow): an equal number of intersections at two distinct times $a<b$ (e.g. two good times where we already know graphicality) implies a constant intersection number throughout the interval $[a,b]$, therefore extending graphicality from pairs of good times to all the times in between.

Similarly to the proof of Theorem 9.2 in \cite{whi-size}, given the sheeting results, the final contradiction to higher multiplicity comes from considering the maximal difference in height between graphical sheets, appropriately rescaled, so that it has value 1 at time $\tau=-1$ and approaches $0$ as $\tau\uparrow 0$. By direct comparison (cf. Lemma \ref{lem:collapse}), this behaviour is impossible for positive solutions of linear parabolic PDEs converging to the heat equation, completing the contradiction.
Moreover, by convexity of the domain, the tangent line to the boundary at a boundary point is a static solution of curve shortening flow which lies strictly on one side of the original flow, so a similar mechanism also rules it out as a (quasi)-static tangent flow, completing the classification.

Finally, the local regularity theorems of White \cite{whi05localreg} and Edelen \cite[Corollary~8.8]{edelen18} exclude multiplicity-one interior lines and orthogonal half-lines at singular points and therefore conclude the proof of Theorem \ref{thm:main}. 

\subsection*{Acknowledgements}
The authors would like to thank Robert Haslhofer for suggesting this project. YB has been supported by a Blyth Fellowship, a Mary H. Beatty Fellowship, a Department of Mathematics Graduate Program Award and an International Graduate Student Scholarship in Mathematics from the University of Toronto.

\section{Preliminaries}\label{sec:preliminaries}
Throughout this paper, $\Omega\subset\Rbb^2$ is a nonempty bounded open convex domain with smooth boundary $S=\partial\Omega$, and $\nu_S$ is its outward unit normal. Consider a smooth family of proper embeddings
\[
\gamma:[0,1]\times[0,T)\longrightarrow\oOm,
\qquad \gamma_t(x)=\gamma(x,t),\qquad \Gamma_t=\gamma_t([0,1]).
\]
Properness means that $\gamma_t((0,1))\subset\Omega$ and $\Gamma_t\cap S=\gamma_t(\{0,1\})$. Let $T_t=\partial_x\gamma/|\partial_x\gamma|$ be the unit tangent vector to $\gamma_t$.

\begin{definition}
	The family $\Gamma=\{\Gamma_t\}_{t\in[0,T)}$ is a \textit{free boundary curve shortening flow} in $\Omega$ if
	\begin{equation}\label{eq:definition-FBCSF}
		\begin{cases}
			\partial_t\gamma(x,t)=-\kappa(x,t)\nu(x,t),
			&x\in[0,1],\ t\in(0,T),\\
			\gamma_t(\{0,1\})\subset S,\quad \nu\perp\nu_S,
			&x\in\{0,1\},\ t\in[0,T).
		\end{cases}
	\end{equation}
	where $\kappa(x,t)$ denotes the curvature of $\Gamma_t$ at $\gamma(x,t)$ with respect to the unit normal field $\nu$, and $\nu_S$ is the outward unit normal to $S$. Notice that the evolution equation can be written as $\partial_t\gamma= \partial^2_s\gamma$, where $s$ denotes the arclength parameter.
\end{definition}
Write $p_-(t)=\gamma(0,t)$ and $p_+(t)=\gamma(1, t)$ for the endpoints of $\Gamma_t$. The orthogonality condition in \eqref{eq:definition-FBCSF} gives
\begin{equation}\label{eq:endpoint-tangents}
	T_t(p_-(t))=-\nu_S(p_-(t)),\qquad
	T_t(p_+(t))=\nu_S(p_+(t)).
\end{equation}
Thus, the outward unit conormal of $\Gamma_t$ is given by $\nu_S$ at both endpoints. 

Stahl’s local existence and uniqueness theory  \cite[Section 2]{stahl96} shows that, given a smooth properly embedded arc meeting $S$ orthogonally at its endpoints as the initial condition, there is always a unique properly embedded solution of \eqref{eq:definition-FBCSF} on $[0, T)$ for some $T>0$. Moreover, under the assumption of convexity of the domain $\Omega$, such a solution is \emph{proper} for as long as it exists. 

We next give a simple proof of preservation of properness in our convex planar setting; see \cite[Corollary~4.5, p.~395]{stahl96} for a more general result valid in any dimension.
\begin{lemma}[Preservation of properness]\label{lem:properness}
Let $\gamma:[0,1]\times[0,T)\to\Rbb^2$ be a family of smooth embeddings satisfying \eqref{eq:definition-FBCSF}. If $\gamma_0((0,1))\subset\Omega$, then $\gamma_t((0,1))\subset\Omega$ for every $t<T$. In particular, $\gamma_t^{-1}(S)=\{0,1\}$.
\end{lemma}
\begin{proof}
Fix $q\in S$ and set $h_q(x,t)=\langle q-\gamma(x,t),\nu_S(q)\rangle$. Since $\partial_t\gamma=\partial_s^2\gamma$, where $s$ denotes arclength, we have
\[
(\partial_t-\partial_s^2)h_q=0.
\]
By convexity of the domain $\Omega$, we have $h_q(0,t),h_q(1,t)\ge0$ for all $t$, while initial properness gives $h_q(x,0)>0$ for $x\in(0,1)$. On each compact time interval contained in $[0,T)$, the equation is uniformly parabolic, so, by the maximum principle, $h_q\ge0$ everywhere and $h_q(x,t)>0$ for $x\in(0,1)$ and $t>0$. Intersecting the supporting half-planes gives $\Gamma_t\subset\oOm$. If $\gamma(x,t)=q\in S$ for some $x\in(0,1)$, then $h_q(x,t)=0$, a contradiction. Thus $\gamma_t((0,1))\subset\Omega$.
\end{proof}
By uniqueness and Lemma \ref{lem:properness}, the solution extends to a maximal properly embedded flow on $[0,\Tmax)$, with $\Tmax\in(0,+\infty]$. Moreover, if $\Tmax<+\infty$, the continuation criterion in \cite[Theorem~1.3]{stahl96} gives
\begin{equation}\label{eq:curvature-blowup}
	\|\kappa(\cdot, t)\|_{L^\infty(\Gamma_t)}^2\to +\infty \text{ as } t\uparrow\Tmax.
\end{equation}

The first variation of length and the boundary orthogonality condition in \eqref{eq:definition-FBCSF} imply that for all $t\in(0, \Tmax)$
\begin{equation}\label{eq:derivative-length}
	\frac{d}{dt}\Hcal^1(\Gamma_t)=-\int_{\Gamma_t}\kappa^2\, ds\leq0.
\end{equation}
Hence, length is non-increasing along the flow.

In order to study singularities of the flow, it is useful to have a notion of weak solution of \eqref{eq:definition-FBCSF}. For this purpose we shall use the framework of \textit{free boundary integral Brakke flows}. 
We refer to \cite{edelen18} for the precise definition and properties of free boundary integral Brakke flows and to \cite{bao} for further foundational results and comparisons with other weak formulations.

If $\Gamma=\{\Gamma_t\}_{t\in[0,T)}$ is a free boundary curve shortening flow in $\Omega$ and we let $\mu_t\coloneqq \Hcal^1\mres\Gamma_t$, then $\{\mu_t\}_{t\in [0,T)}$ is a 1-dimensional free boundary integral Brakke flow in $\Omega$. For all $t\in[0,T)$, the integral 1-varifold $V_t$ associated to $\mu_t$ is simply the multiplicity-one varifold $|\Gamma_t|\coloneqq \vbf(\Gamma_t, \uno_{\Gamma_t})$ induced by $\Gamma_t$. To keep the notation simple, we shall use $\Gamma$ to denote both the smooth flow and the corresponding Brakke flow.

As mentioned in the introduction, we are going to make use of the following result, which says that the number of intersections of two compact free boundary curve shortening flows is non-increasing in time. By writing the flows locally as graphs, and seeing their intersections as zeros of the graphical difference, which satisfies a linear parabolic equation as in \cite{Angenent_survey}, it follows directly from the Sturm zero-number theorem in \cite{Angenent88} (see \cite{Lou19} for the boundary adaptation needed for free boundary flows).
\begin{lemma}[Intersection principle]\label{lem:zero}
	Let $I$ be an open time interval.
	Let $\Gamma_t^1,\Gamma_t^2\subset\Rbb^2$ be smooth families of compact embedded arcs with curvature-vector normal velocity. Each endpoint either lies on $S$ for every $t\in I$, with the arc orthogonal to $S$ and its adjacent subarc in $\Omega$, or remains disjoint from the other arc for every $t\in I$.
	Suppose also that
	\[
	\partial\Gamma_t^1\cap(\Gamma_t^2)^\circ
	=\partial\Gamma_t^2\cap(\Gamma_t^1)^\circ=\emptyset,
	\quad t\in I.
	\]
	Assume that the arcs share no subarc of positive length at any time. Then the total number of intersections
	\[
	N(t)=\#(\Gamma_t^1\cap\Gamma_t^2)
	\]
	is finite and non-increasing, and decreases strictly across an interior tangency or a common endpoint. 
	
	In particular, if all intersections at $t_1,t_2\in I$, $t_1<t_2$, are transverse and interior, and $N(t_1)=N(t_2)$, then every intersection on $[t_1,t_2]$ is transverse and interior, and $N$ is constant there.
\end{lemma}

%
%
%

\section{The immortal case}\label{sec:immortal}
We begin with the case when $\Tmax=\infty$. In this case, the flow converges smoothly to an orthogonal chord. The proof of subsequential convergence is rather standard; however, the proof of uniqueness of the limit is more involved; the key ingredient in the uniqueness argument is the elementary estimate in \eqref{eq:main-estimate-uniqueness}.

\begin{proposition}\label{prop:immortal}
	If $\Tmax=\infty$, then $\Gamma_t$ converges smoothly to a single orthogonal chord.
\end{proposition}
\begin{proof}
	For simplicity, let $L(t)\coloneqq \mathcal{H}^1(\Gamma_t)$ for $t\in[0, \Tmax)$. Since length is non-increasing along the flow by \eqref{eq:derivative-length},  $L(t)$ has a limit $L_\infty\ge0$ as $t\to \infty$.
	
	We first notice that the limit length $L_\infty$ cannot be 0. Indeed, if $L_\infty$ were zero, then, for large $t$, the diameter of $\Gamma_t$ would be arbitrarily small, and $\Gamma_t$ would be forced to go extinct in finite time by comparison with a sufficiently small shrinking half-circle. Hence $L_\infty>0$.
	
	Now let $t_i\to\infty$ be arbitrary and set $\Gamma^i_\tau=\Gamma_{t_i+\tau}$. Since $L(t)\to L_\infty$, \eqref{eq:derivative-length} implies
	\begin{equation}\label{eq:immortal-translate}
		\int_{-2}^{-1}\int_{\Gamma^i_\tau}\kappa^2\,ds\,d\tau
		=L(t_i-2)-L(t_i-1)\to 0.
	\end{equation}
	By a standard real analysis argument, there exists $\tau_*\in(-2,-1)$ with $\int_{\Gamma_{t_i+\tau_*}}\kappa^2\,ds\to0$. After constant-speed reparametrisation, we have
	\[
	|\gamma_i'|=L(t_i+\tau_*),\quad
	\int_0^1|\gamma_i''|^2\,dx
	=L(t_i+\tau_*)^3
	\int_{\Gamma_{t_i+\tau_*}}\kappa^2\,ds\to 0.
	\]
	By the compact embedding $W^{2,2}(I)\Subset C^{1,\alpha}(I)$, $0<\alpha<1/2$, there is a subsequence of $(\gamma_i)_i$ converging to an affine parametrisation of a segment $\sigma$ of length $L_\infty$. Orthogonality passes to the limit. Thus $\sigma$ is a proper orthogonal chord. By parabolic regularity \cite{Lieberman},
	\begin{equation}\label{eq:immortal-subsequence}
		\Gamma_{t_i}\to\sigma\quad\text{smoothly along a subsequence}.
	\end{equation}
	
	It remains to prove uniqueness of the limiting chord. 
	Let $K=\bigcap_{T>0}\overline{\{p_-(t):t\ge T\}}\subset S$ be the set of all possible limit left endpoints and notice it is nonempty, compact, and connected. The key estimate (somewhat inspired by arguments of Lauer in \cite{Lauer13}) is the following:  if $p_-(t)\in K$ for all $t\in [a,b]\subset[0, \Tmax)$, then
	\begin{equation}\label{eq:main-estimate-uniqueness}
		\int_a^b\sqrt{-L'(t)}\, dt\leq \sqrt{2} L(0)\sqrt{L(a)-L_\infty}.
	\end{equation}
	To prove \eqref{eq:main-estimate-uniqueness}, for every $q\in K$, consider the orthogonal chord
	\[
	\sigma_q=[q,r_q],\quad
	r_q=q+L_\infty e(q),\quad e(q)=-\nu_S(q)=\nu_S(r_q).
	\]
	Whenever $p_-(t)=q\in K$, by looking at the tangent line to $S$ at $r_q$, we have by convexity 
	\[
	\langle p_+(t)-q,e(q)\rangle\le L_\infty.
	\]
	Therefore, since $|T_t(s)-T_t(0)|^2=2-2\langle T_t(s), T_t(0)\rangle$, we have
	\begin{equation}\label{eq:ODE-excess}
		L(t)-L_\infty \leq L(t)-\langle p_+(t)-q,e(q)\rangle=\frac12\int_0^{L(t)}|T_t(s)-T_t(0)|^2\,ds
		\leq -\frac{L(0)^2}{2} L'(t),
	\end{equation}
	where the last inequality follows from Cauchy--Schwarz. Integrating \eqref{eq:ODE-excess}, we obtain \eqref{eq:main-estimate-uniqueness}.
	
	We shall now show that if $K$ contains a nontrivial arc, then \eqref{eq:main-estimate-uniqueness} leads to a contradiction. Choose $q_0$ in the relative interior of $K$ and consider the orthogonal chord $\sigma_{q_0}=[q_0, r_0]$ starting at $q_0$; then fix a sufficiently small relatively open arc $J$ containing $q_0$, with $\overline J$ contained in the relative interior of $K$ and disjoint from $r_0$. By definition of $K$, there exists $a_i \to \infty$ such that $p_-(a_i)\to q_0$, and, after passing to a subsequence, $\Gamma_{a_i}\to \sigma_{q_0}$ smoothly. For all large $i$, let $b_i$ be the first time after $a_i$ such that $p_-(t)$ leaves $J$ and notice that $b_i$ is finite because $K\not\subset \overline{J}$. Then $p_-(b_i)\in \partial J$. 
	
	We will show that \eqref{eq:main-estimate-uniqueness} forces the curve's squared-height integral above $\sigma_{q_0}$ to tend to zero at $b_i$, whereas exiting $J$ implies a fixed positive lower bound, providing the desired contradiction. Fix a unit normal $n_0$ to the line containing $\sigma_{q_0}$, and set
	\[
	h(\cdot)=\langle\cdot-q_0, n_0\rangle,\quad F(t)=\int_{\Gamma_t}h^2\, ds.
	\]
	Since $\Gamma_{a_i}\to \sigma_{q_0}$ smoothly, $F(a_i)\to0$. By the orthogonality condition at the endpoints and Cauchy--Schwarz, we have
	\begin{equation}\label{eq:height-evolution}
		F'(t) =-2\int_{\Gamma_t}h\kappa\langle n_0, \nu\rangle\, ds -\int_{\Gamma_t}h^2\kappa^2\, ds\leq-2\int_{\Gamma_t}h\kappa\langle n_0, \nu\rangle\, ds\leq  2\sqrt{F(t)}\sqrt{-L'(t)}.
	\end{equation}
	Therefore, $F(b_i)^{1/2}\leq F(a_i)^{1/2}+\int_{a_i}^{b_i}\sqrt{-L'(t)}\, dt\to 0$ by \eqref{eq:main-estimate-uniqueness}.
	
	On the other hand, if $\tilde\gamma_{b_i}$ is an arclength parametrisation of $\Gamma_{b_i}$ starting at $p_-(b_i)$, there is $\delta>0$ such that $|h(p_-(b_i))|\geq\delta$, $|\partial_s h|\le1$, and $L(b_i)\ge L_\infty\ge\delta$, so that $|h(\tilde\gamma_{b_i}(s))|\ge\delta-s\ge\frac\delta2$ for $0\le s\le\delta/2$. Consequently,
	\[
	F(b_i)\ge\int_0^{\delta/2}\frac{\delta^2}{4}\, ds
	=\frac{\delta^3}{8}>0,
	\]
	contradicting $F(b_i)\to0$. This provides the desired contradiction.
	
	Thus $K=\{q_0\}$ and $p_-(t)\to q_0$. Every subsequential limiting chord is then $\sigma_{q_0}$, which shows the desired uniqueness.
\end{proof}

\section{Classification of tangent flows}\label{sec:classification}
Let us now assume $\Tmax<+\infty$. We are going to study the behaviour of the flow at the first singular time $\Tmax$. The first step is to analyse the tangent flows on which the flow is modelled.

We write $\spt\Gamma$ for the closure of the spacetime track, including its limit points at $\Tmax$. Let $X_0=(p_0, \Tmax)\in\spt\Gamma$. 
For $\lambda_i\to+\infty$, define the parabolically rescaled flows $\Gamma^i=\{\Gamma^i_\tau\}_{\tau\in[-\lambda_i^2\Tmax, 0)}$ by
\begin{equation}\label{eq:rescale}
	\Gamma^i_\tau\coloneqq \lambda_i(\Gamma_{\Tmax+\lambda_i^{-2}\tau}-p_0).
\end{equation}
Correspondingly, rescale the domain $\Omega$ so that $\Omega_i=\lambda_i(\Omega-p_0)$ and $S_i=\partial \Omega_i$. If $p_0\in\Omega$, then $\Omega_i\to\Rbb^2$ locally. If $p_0\in S$, fix Euclidean coordinates $(x,y)$ so that
\[
H=\{x\ge0\},\quad P=\partial H=\{x=0\},\quad
L_+=\{y=0,\ x\ge0\},
\]
and $(\overline\Omega_i,S_i)\to(H,P)$ locally smoothly. In these coordinates, $\Sbb^1_+=\Sbb^1\cap H$.

By Edelen's compactness and monotonicity results \cite[Lemma 6.1, Proposition 6.2, Theorem 6.4]{edelen18}, up to a subsequence (not relabelled), $\Gamma^i\to \Gamma^\infty$ in the sense of Brakke flows, where $\Gamma^\infty$ (called a \emph{tangent flow} at $X_0$) is
\begin{itemize}
	\item either a 1-dimensional self-shrinking integral Brakke flow in $\Rbb^2$ (if $p_0\notin S$) 
	\item or a 1-dimensional free boundary integral Brakke flow in the half-space $H$ bounded by $P$ (if $p_0\in S$), such that its reflection $\widetilde{\Gamma}^\infty$ about $P$ is a self-shrinking integral Brakke flow in $\Rbb^2$. 
\end{itemize}

We denote by $\Sigma_\tau$ the spatial support of the time-$\tau$ slice of $\Gamma^\infty$, without multiplicity. If the tangent flow is static for $\tau<0$, we suppress the time index and write $\Sigma_\tau=\Sigma$. Throughout the paper, classification and uniqueness of tangent flows concern their negative-time restrictions.

Edelen's results hold for free boundary integral Brakke flows in any dimension (and codimension). In the present one-dimensional setting, the uniform $L_\mathrm{loc}^2$ bound on the curvature of the curves in $\Gamma^i$ (which is a consequence of Edelen's monotonicity formula \cite[Theorem 5.1]{edelen18}) is rather powerful due to the Sobolev embedding $W^{2,2}(I)\hookrightarrow C^{1, 1/2}(I)$, and allows us to obtain the following classification result (see also \cite[Lemma 6.2]{lz23}). 

\begin{proposition}[Tangent-flow classification]\label{prop:tangent-flow-classification}
	Let $X_0=(p_0,\Tmax)\in\spt\Gamma$. Every tangent flow $\Gamma^\infty$ at $X_0$ has negative-time slices $q|\Sigma_\tau|$ for a fixed positive integer $q$. If $p_0\in\Omega$, then $\Sigma_\tau=\Sigma$ is a line through the origin. If $p_0\in S$, then one of the following holds:
	\begin{enumerate}
		\item[\textup{(B1)}] the shrinking semicircle $\Sigma_\tau=\sqrt{-2\tau}\,\Sbb^1_+$ in $H$, with $q=1$;
		\item[\textup{(B2)}] the static orthogonal half-line $\Sigma=L_+$;
		\item[\textup{(B3)}] the static boundary line $\Sigma=P$.
	\end{enumerate}
	Moreover, if $p_0\in S$, the tangent flow at $X_0$ is unique.
\end{proposition}

\begin{proof}
	As mentioned above, up to a subsequence (not relabelled), $\Gamma^i$ converge as Brakke flows to some self-shrinking free boundary Brakke flow $\Gamma^\infty$. This limit is nonzero by \cite[Theorem~6.4 and Corollary~7.3]{edelen18}.
	
	By the proof of Edelen's monotonicity formula \cite[Theorem 5.1]{edelen18}, there is a set $G\subset(-\infty,0)$ of full measure such that, for every $\tau\in G$ and $R>0$,
	\begin{equation}\label{eq:defect}
		\begin{split}
			\sup_i\int_{\Gamma^i_\tau\cap B_R}(1+\kappa_i^2)\,ds<\infty,\quad \int_{\Gamma^i_\tau\cap B_R}\left|-\kappa_i\nu_i-\frac{p^\perp}{2\tau}\right|^2\,ds\to 0.
		\end{split}
	\end{equation}
	where $\kappa_i$ denotes the curvature of the smooth embedded curve $\Gamma_\tau^i$ and $\nu_i$ its unit normal. By the Sobolev embedding $W^{2,2}(I)\hookrightarrow C^{1, 1/2}(I)$, for every $\tau \in G$ and $\alpha\in(0, 1/2)$, up to a subsequence (depending on $\tau$), the maps $\gamma^i_\tau$ converge in $C^{1, \alpha}_{\mathrm{loc}}$ to a proper and connected smooth immersion $\gamma^\tau_\infty$. Moreover, by passing to the limit in \eqref{eq:defect}, $\gamma^\tau_\infty$ is a weak solution of the shrinker equation
	\begin{equation}\label{eq:shrinker-equation}
		\kappa-\frac{p\cdot \nu}{-2\tau}=0
	\end{equation}
	and, if $p_0\in S$, then $\gamma^\tau_\infty$ satisfies the Neumann condition in \eqref{eq:definition-FBCSF} on the limit boundary $P$, so that its image $\Sigma_\tau$ meets $P$ orthogonally at any boundary points. By elliptic regularity, $\Sigma_\tau$ is a smooth immersed curve and it is locally a $C^{1, \alpha}$ limit of embedded curves, and therefore it cannot have any transverse self-intersections.
	
	By the classification in \cite{al86}, the only smooth immersed planar self-shrinkers with no transverse self-intersections are lines through the origin and the circle $\sqrt{-2\tau}\,\Sbb^1$. Since every embedded free boundary solution in $H$ gives rise to a smooth embedded solution in $\Rbb^2$ after reflection about $P$, we conclude the following.
	\begin{itemize}[leftmargin=20pt]
		\item \emph{The interior case.} If $p_0\notin S$,  
		\begin{itemize}[leftmargin=36pt]
			\item[\textup{(I1)}] if $\Sigma_\tau$ is compact, then $\Sigma_\tau=\sqrt{-2\tau}\,\Sbb^1$;
			\item[\textup{(I2)}] if $\Sigma_\tau$ is non-compact, then $\Sigma_\tau=\Sigma$ is a line through the origin.
		\end{itemize}
		\item \emph{The boundary case.} If $p_0\in S$,
		\begin{itemize}[leftmargin=36pt]
			\item[\textup{(B1)}] if $\Sigma_\tau$ is compact, then $\Sigma_\tau=\sqrt{-2\tau}\,\Sbb^1_+$;
		\end{itemize}
		and, if $\Sigma_\tau$ is non-compact, then $\Sigma_\tau=\Sigma$, where
		\begin{itemize}[leftmargin=36pt]
			\item[\textup{(B2)}] either $\Sigma=L_+$, the orthogonal half-line, 
			\item[\textup{(B3)}] or $\Sigma=P$, the boundary line.
		\end{itemize}
	\end{itemize}

	Let us now rule out \textup{(I1)}. Let $R>\sqrt{-2\tau}$ and notice that if $\Sigma_\tau=\sqrt{-2\tau}\,\Sbb^1$, then local Hausdorff convergence of the spacetime supports gives $\Gamma^i_\tau\cap\partial B_R(0)=\varnothing$ for all large $i$; on the other hand, since $\Gamma^i_\tau$ is connected, this means that $\Gamma^i_\tau\subset B_R(0)$ for all large $i$, contradicting the fact that $\Gamma^i_\tau$ is an arc with endpoints on $S_i$, and $\dist(0,S_i)=\lambda_i\dist(p_0,S)\to\infty$.
	
	Next, let us rule out higher multiplicity in case \textup{(B1)}. Let $\tau \in G$ and let $\ell=\Hcal^1(\Sigma_\tau)=\pi\sqrt{-2\tau}$ be the length of the limiting semicircle. Similarly to the proof of case \textup{(I1)}, for any large $i$ the whole of $\Gamma^i_\tau$ lies in a fixed ball. Therefore, as $\tau$ is a good time, up to a subsequence, the constant-speed parametrisations $\tilde \gamma_i:[0,1]\to \Gamma^i_\tau$ converge in $C^{1, \alpha}$ to an immersion $\tilde \gamma_\infty$ of constant speed $q\ell$ with image $\Sigma_\tau$. Now express the limiting map in the arclength coordinate of the semicircle, namely, take an arclength parametrisation $\alpha:[0,\ell]\to\Sigma_\tau$ and let $f\coloneqq \alpha^{-1}\circ\tilde \gamma_\infty:[0,1]\to[0, \ell]$. Notice that $|f'|=q\ell>0$, so that $f'$ has a fixed sign. Thus,
	\[
	0<q\ell=\int_0^1|f'|\,dx=|f(1)-f(0)|\le\ell,
	\]
	so that $q=1$, as required.
	
	Finally, we prove that boundary tangent flows are unique. The argument is an instance of a general principle: the set of limits of a precompact continuous connected family of rescalings is connected, so that, if it is finite, then it has to be a singleton. For $\lambda\ge\lambda_0$ with $\Tmax-\lambda_0^{-2}>0$, let
	\[
	\zeta_\lambda=\Hcal^1\mres\bigl[\lambda(\Gamma_{\Tmax-\lambda^{-2}}-p_0)\bigr].
	\]
	This is the time-$(-1)$ measure of the original flow $\Gamma$, parabolically rescaled about $X_0$ by the factor $\lambda$. By Edelen's local mass bound \cite[Lemma~6.1]{edelen18}, every set $\{\zeta_\lambda:\lambda\geq a\}$, for $a\ge\lambda_0$, has compact closure in the local weak topology of Radon measures, and \cite[Theorem~4.14 and Proposition~6.2]{edelen18} identify $ \Kcal=\bigcap_{a\ge\lambda_0}\overline{\{\zeta_\lambda:\lambda\geq a\}}$ as the set of all possible time-$(-1)$ tangent measures. Notice that by definition $\Kcal$ is nonempty, compact, and connected.
	
	On the other hand, we have already shown that
	\[
	\Kcal\subset \bigl\{q\Hcal^1\mres P,\; q\Hcal^1\mres L_+,\; \Hcal^1\mres(\sqrt2\,\Sbb^1_+) :q\in\N,\ q\ge1\bigr\},
	\]
	i.e. each measure in $\Kcal$ is induced by one of $P$ and $L_+$ with some positive integer multiplicity $q$, or $\sqrt2\,\Sbb^1_+$ with multiplicity one. The local mass bound \cite[Lemma~6.1]{edelen18} bounds the possible multiplicities, so $\Kcal$ is finite, and its connectedness implies $\Kcal$ is a singleton.  Backward self-similarity determines every negative-time slice from
    the slice at $-1$, proving uniqueness for $\tau<0$. This concludes the proof.
\end{proof}

\begin{remark}[Good times]\label{rmk:good-times}
	We call the full-measure set $G\subset (-\infty, 0)$ selected in the proof of Proposition \ref{prop:tangent-flow-classification} the set of \emph{good times}. At a good time $\tau$,  \eqref{eq:defect} is satisfied and, up to a subsequence (depending on $\tau$), the rescaled curves $\Gamma^i_\tau$ converge to $\Sigma_\tau$ in the local $C^{1,\alpha}$ sense described above. In varifold notation, $|\Gamma^i_\tau|\rightharpoonup q|\Sigma_\tau|$, where $q$ is the multiplicity in Proposition~\ref{prop:tangent-flow-classification}.
\end{remark}

\section{Graphical convergence to a static tangent flow}

Our main sheeting result (Proposition~\ref{prop:sheeting}) is analogous in spirit to White's sheeting theorem \cite[Theorems~8.1--8.2]{whi-size}, used in his structure theory for mean-convex flow \cite{whi-nature}, and to the free boundary sheeting results of Edelen--Haslhofer--Ivaki--Zhu \cite[Theorem~6.10 and Corollaries~6.12--6.14]{FB-mean-convex}: weak convergence to a (possibly higher multiplicity) static tangent flow is upgraded to smooth convergence of separate graphical sheets. Here no sign condition is imposed on the curvature of the evolving arc. Following White's curve-shortening argument \cite[Theorem~7.2 and Lemma~7.4]{White_lectures}, we first obtain graphicality at \emph{good times} by one-dimensional compactness, then extend it between those times using the intersection principle of Lemma~\ref{lem:zero}.

Consider static tangent flows, with spatial support $\Sigma$ as in Section~\ref{sec:classification}. When $\Sigma$ is a line, we write $\graph_i u$ for the Euclidean normal graph of the function $u$ on the line, choosing the normal into $H$ when $\Sigma=P$. When $\Sigma=L_+$, we identify it with $[0,\infty)$ by its $x$-coordinate and use the inward Fermi chart
\begin{equation}\label{eq:fermi-chart}
	\Phi_i(x,y)=\beta_i(y)-x\nu_{S_i}(\beta_i(y)),\qquad x\ge0,
\end{equation}
where $\beta_i$ parametrises $S_i$ by arclength with $\beta_i(0)=0$ and $\beta_i'(0)=(0,1)$, and set $\graph_i u=\{\Phi_i(x,u(x)):x\in\dom u\}$. These charts converge locally smoothly to the identity.

\begin{proposition}[Sheeting near a static tangent flow]\label{prop:sheeting}
	Suppose the rescalings \eqref{eq:rescale} converge to a static tangent flow with negative-time slices $q|\Sigma|$. For every $R>0$ and $I=[a,b]\Subset(-\infty,0)$, $a<b$, and all large $i$, there are smooth ordered functions $u_i^1<\cdots<u_i^q$ on $(\Sigma\cap B_{2R}(0))\times I$ such that
	\[
	\Gamma^i_\tau\cap B_R(0)
	=B_R(0)\cap\bigcup_{j=1}^q\graph_i u_i^j(\cdot,\tau)
	\qquad(\tau\in I),
	\]
	and $u_i^j\to0$ in $C^\infty((\Sigma\cap\overline{B}_R(0))\times I)$, including $\partial\Sigma$ when present.
	
	If $\Sigma$ is a line, $B_R(0)$ contains no endpoints; if $\Sigma=P$, moreover, $u_i^j>0$. If $\Sigma=L_+$, each graph starts at a distinct endpoint of $\Gamma^i_\tau$ and satisfies $\partial_xu_i^j(0,\tau)=0$; in particular, $q\le2$.
\end{proposition}
\begin{proof}
	Fix $R>0$ and $I=[a,b]\Subset(-\infty,0)$ and choose good times $a'<a<b<b'<0$. Let $I'=[a',b']$. Notice that, by local Hausdorff convergence of the spacetime supports \cite[Theorem~4.14 and Corollary~7.3]{edelen18}, for every compact $K\subset\Rbb^2$, we have
	\begin{equation}\label{eq:static-support-control}
		\sup_{\tau\in I'}\sup_{p\in\Gamma^i_\tau\cap K}
		\dist(p,\Sigma)\rightarrow0.
	\end{equation}
		
\textit{Step 1: Sheeting at the good times $a', b'$}.
Let $n$ be a unit normal to the line/half-line $\Sigma$ and let $\tau$ be either one of $a',b'$. Notice that, since $\tau$ is a good time, by $C^{1, \alpha}$-convergence (cf.~ Remark \ref{rmk:good-times}), for every compact $K\subset\Rbb^2$ we have
	\begin{equation}\label{eq:good-time-uniform-tilt}
		\sup_{p\in\Gamma^i_\tau\cap K}|\langle T_i(p),n\rangle|\to 0,
	\end{equation}
where $T_i$ is the unit tangent to $\Gamma^i_\tau$. Let $\rho, h>0$ be fixed.
\begin{itemize}[leftmargin=*]
	\item \textit{The interior case.} Use Euclidean coordinates $(x,y)$ with $\Sigma=\{x=0\}$ and let $Q_{h, \rho}\coloneqq (-h, h)\times(-\rho, \rho)\subset \mathbb{R}^2$. First notice that for all sufficiently large $i$ there are no endpoints in $Q_{h, \rho}$. Indeed, as $p_0\in\Omega$ and the endpoints of $\Gamma^i_\tau$ lie in the boundary $S_i$, for all sufficiently large $i$, they lie outside $Q_{h, \rho}$. By \eqref{eq:static-support-control}, the sides $x=\pm h$ avoid the flow for all large $i$. 
	By \eqref{eq:good-time-uniform-tilt}, on each of the $N_i$ connected components of $\Gamma^i_\tau\cap Q_{h, \rho}$, $|\partial_s y|\geq 1/2$ and therefore the derivative has a fixed sign for all large $i$. Hence, every connected component of $\Gamma^i_\tau\cap Q_{h, \rho}$ joins the two sides $y=\pm \rho$ and is a graph. The components are disjoint since $\Gamma^i_\tau$ is an embedded curve.
	\item \textit{$\Sigma = P$.} The argument is nearly identical to the one in the previous case. The only difference is in how we establish that there are no endpoints of $\Gamma^i_\tau$ in $Q_{h, \rho} = (-h, h)\times(-\rho, \rho)$. In the case $\Sigma=P$, this follows from \eqref{eq:good-time-uniform-tilt}: indeed, if $\Gamma^i_\tau$ had an endpoint in $Q_{h, \rho}$, the orthogonality condition in \eqref{eq:definition-FBCSF} would contradict \eqref{eq:good-time-uniform-tilt}.
	\item \textit{$\Sigma=L_+$.} In $\Phi_i([0,\rho)\times(-h,h))$, support convergence and $\Phi_i\to\mathrm{Id}$ show that the faces $y=\pm h$ avoid the flow for all large $i$. By the same argument as above, the inward distance coordinate satisfies $|\partial_s x|\ge1/2$ for all large $i$. Thus each of the $N_i$ connected components of $\Gamma^i_\tau \cap \Phi_i([0,\rho)\times(-h,h))$ joins $x=0$ to $x=\rho$: the positive lower bound on $|\partial_s x|$ excludes two ends on the same face and properness prevents the curve from touching $x=0$ outside of an endpoint. Thus the components are graphs over $\{0\leq x<\rho\}$.
\end{itemize}
In all cases, $C^0$-convergence of the functions $u_i^1<\dots<u_i^{N_i}$ representing the graphical sheets to $0$ follows from \eqref{eq:static-support-control}; then \eqref{eq:good-time-uniform-tilt} implies $C^1$-convergence. In the case $\Sigma=L_+$, the orthogonality condition at the endpoints implies the homogeneous Neumann condition $\partial_xu_i^j(0,\tau)=0$ for all $j$.
Finally, each graph contributes a uniform positive amount of mass on a smaller fixed rectangle, so the local mass bound bounds $N_i$. Since each graph converges in $C^1$ to the same segment, varifold convergence to $q|\Sigma|$ then forces $N_i=q$ for all sufficiently large $i$. In the case $\Sigma=L_+$, in particular $q\leq 2$, because the graphs start at distinct endpoints and each $\Gamma^i_\tau$ has exactly two endpoints.
	
	\vspace{1ex}
	
	\emph{Step 2: Extending sheeting to the closed time interval $[a', b']$.} As in 	\emph{Step 1}, we shall divide the argument into three cases.
	\begin{itemize}[leftmargin=*]
		\item \textit{The interior case.} The argument for the interior case follows \cite[Theorem~7.2 and Lemma~7.4]{White_lectures}. We record it here for completeness.
		
		As in \emph{Step 1}, use Euclidean coordinates $(x,y)$ with $\Sigma=\{x=0\}$, and fix the rectangles
		\[
		Q=(-1,1)\times(-2R,2R),\quad Q'=(-2,2)\times(-2R-1,2R+1).
		\] 
		Again, as in \emph{Step 1}, for all sufficiently large $i$ (depending on $R$ only) $Q'$ cannot contain any endpoints of $\Gamma^i_\tau$. Let us now apply Lemma \ref{lem:zero} to line segments transverse to $\Sigma$. At times $a',b'$, \emph{Step 1} gives exactly $q$ Euclidean graphs in $Q'$ with small $C^1$-norm. Thus, for each $c\in (-2R-1, 2R+1)$, the segment $\sigma_c=[-2,2]\times\{c\}$ has exactly $q$ transverse intersections with both $\Gamma^i_{a'}$ and $\Gamma^i_{b'}$, and its endpoints avoid the flow for all times $\tau\in I'$ by \eqref{eq:static-support-control}. Hence, by Lemma \ref{lem:zero}, for all $c\in (-2R-1, 2R+1)$
		\[
		\tau \mapsto  \#(\Gamma^i_{\tau}\cap \sigma_c) \  \text{is constant on } I'=[a', b'],
		\]
		and  $\sigma_c$ has exactly $q$ transverse intersections with $\Gamma^i_{\tau}$ throughout $I'$. Ordering these intersections gives $q$ smooth graphs $u_i^1(\cdot, \tau)<\dots<u_i^{q}(\cdot, \tau)$ in $Q'$ throughout $I'$.
		
		Their slopes also tend uniformly to zero in $Q$. To see this, fix $\epsilon\in(0,\pi/2)$ and choose $h>0$ so small that any line through a point of $Q\cap\{|x|\le h\}$ and making an angle of at least $\epsilon$ with $\Sigma$ meets both lines $x=\pm h$ inside $Q'$.
		For large $i$, by \eqref{eq:static-support-control}, $\Gamma^i_\tau\cap Q'\subset\{|x|<h/2\}$ for every $\tau\in I'$. Moreover, by \emph{Step~1}, at both good times the graph slopes have absolute value less than $\tan\epsilon$.
		Suppose that at some intermediate time a tangent line at a point of $\Gamma^i_\tau\cap Q$ makes an angle of at least $\epsilon$ with $\Sigma$. Consider the fixed segment of this line between $x=-h$ and $x=h$. At each good time, it meets every sheet exactly once and transversely. Thus its total intersection number is $q$ at both $a'$ and $b'$. Since $Q'$ contains no curve endpoints, Lemma~\ref{lem:zero} excludes any tangential intersection, giving the required contradiction. Thus the tangent lines in $Q$ converge uniformly to the direction of $\Sigma$ throughout $I'$.
		
		\item \textit{$\Sigma =P$.} Again, use Euclidean coordinates $(x,y)$ with $\Sigma=\{x=0\}$, and fix the rectangles
		\[
		Q=(-1,1)\times(-2R,2R),\quad Q'=(-2,2)\times(-2R-1,2R+1).
		\] 
		We shall prove that for all sufficiently large $i$ the larger rectangle $Q'$ does not contain any endpoints of $\Gamma^i_\tau$ for any $\tau \in I'$. Then the interior argument of the previous paragraph applies verbatim.
		
		In order to rule out $\Gamma^i_\tau$ having any endpoints in $Q'$, we again use Lemma \ref{lem:zero} with segments normal to $S_i$ foliating a small collar neighbourhood of $S_i$ around the origin. Parametrise $S_i$ by $\beta_i$ as in \eqref{eq:fermi-chart} and for $|c|\le2R+2$ consider the segments $\ell_{i,c}=\{\beta_i(c)-s\nu_{S_i}(\beta_i(c)):0\le s\le1\}$.
		As $S_i\to P$ smoothly, for large $i$, they lie in a uniform collar neighbourhood of $S_i$ and meet $S_i$ only at their endpoints at $s=0$. By \eqref{eq:static-support-control}, their endpoints at $s=1$ remain disjoint from the flow. At both good times,
		\begin{equation}\label{eq:preservation-intersection}
			\#(\Gamma^i_{a'}\cap\ell_{i,c})
			=\#(\Gamma^i_{b'}\cap\ell_{i,c})=q,
		\end{equation}
		with all intersections transverse and interior, since the graphs and $S_i$ are $C^1$-close to $P$.
	
		If, for some $\tau$ in $(a', b')$, the curve $\Gamma^i_\tau$ had an endpoint at $\beta_i(c)$, then $\Gamma^i_\tau$ would be tangent to $\ell_{i,c}$ at $\beta_i(c)$ and Lemma \ref{lem:zero} would force a strict decrease in the total intersection number between $a'$ and $b'$, contradicting \eqref{eq:preservation-intersection}.

        Since $P=\{x=0\}$ is a supporting line of the convex domain $\Omega_i$, with $\Omega_i\subset\{x>0\}$, properness and the absence of endpoints in $Q'$ imply that every sheet lies strictly in $\{x>0\}$; hence $u_i^j>0$ throughout $I'$.
		
		\item \textit{$\Sigma = L_+$.} Finally, in the half-line case, we first apply the interior argument away from $S_i$ and then use the maximum principle to propagate graphicality all the way to the endpoints.
		
		Consider the fixed collar $\Phi_i([0,2R]\times(-1,1))$. Applying the interior argument in a neighbourhood of the outer side $\Phi_i(\{2R\}\times(-1,1))$ gives exactly $q$ transverse crossings, with tangent lines tending uniformly to the horizontal. At $a'$, \emph{Step~1} identifies the curve in the collar as $q$ subarcs, each joining an endpoint to the outer side. Since the crossings remain transverse, their order along the curve cannot change; this preserves the decomposition throughout $I'$.
		
		Orient each subarc from its boundary endpoint towards the outer side, and set $f^i=\langle T_i,e_x\rangle$, where $e_x=(1,0)$. At the intersection with the outer side, by the above, $f^i\to 1$. The same holds at the boundary endpoint, because of orthogonality. Together with \eqref{eq:good-time-uniform-tilt} at $a'$, these estimates give constants $c_i>0$, with $c_i\to1$, such that $f^i\ge c_i$ initially and at both ends throughout $I'$. Writing $s$ for arclength and $\eta$ for the tangential parametrisation speed, we have
		\[
		f^i_\tau=f^i_{ss}+\eta f^i_s+\kappa_i^2f^i.
		\]
		Positive constants are subsolutions, so the maximum principle gives $f^i\ge c_i$ along every subarc throughout $I'$. Since the gradient of the inward distance coordinate converges uniformly to $e_x$ in the collar, that coordinate is strictly increasing along each subarc for all large $i$.
		
		Thus the subarcs are the graphs $\graph_i u_i^j$ of smooth ordered functions $u_i^1<\cdots<u_i^q$ on $[0,2R)\times I'$. Their heights tend uniformly to zero by \eqref{eq:static-support-control}, and $\partial_xu_i^j=(\partial_s y)/(\partial_s x)\to0$ since $f^i\geq c_i\to 1$ and $\Phi_i$ converges smoothly to the identity. At the endpoints, orthogonality gives $\partial_xu_i^j(0,\tau)=0$.
	\end{itemize}

	\emph{Smooth convergence.}
	Finally, we upgrade $C^1$ convergence to smooth convergence by standard parabolic regularity. If $\Sigma$ is a line, then the functions $u_i^j$ solve the equation $u_\tau=u_{yy}/(1+u_y^2)$. \\
	If $\Sigma = L_+$, then in  Fermi coordinates \eqref{eq:fermi-chart}, we have
	$\Phi_i^*\delta=dx^2+a_i(y,x)^2dy^2$, where $a_i(y,x)\rightarrow1$ locally smoothly.
	Therefore, the functions $u_i^j$ solve a parabolic equation with principal coefficient $(1+a_i^2u_x^2)^{-1}$, and homogeneous Neumann data at $x=0$. The $C^1$ bounds  give uniform parabolicity, so interior and boundary estimates \cite{Lieberman} give smooth convergence to zero as required. 
\end{proof}

\begin{remark}[Convexity]\label{rmk:convexity}
    Convexity of the domain $\Omega$ is used crucially in the case $\Sigma= P$ to obtain that $u_i^j(\cdot, \tau)>0$ for all $\tau \in I$. Indeed, in our chosen coordinate system, by convexity, $\Omega_i\subset \{x>0\}$. If $\Sigma = P$ and $Q\subset \mathbb{R}^2$ is a compact rectangle containing the origin in its interior, the proof of Proposition \ref{prop:sheeting} shows that for all sufficiently large $i$, the rescaled curves $\Gamma^i_\tau$ have no endpoints in $Q$ for all $\tau \in I$. Therefore, by properness, $\Gamma^i_\tau\cap Q\subset \Omega_i\subset \{x>0\}$ for all sufficiently large $i$ and all $\tau \in I$.
\end{remark}

\section{Multiplicity one and exclusion of the boundary line}\label{sec:multiplicity}

In this section, we finally rule out the remaining tangent flows, namely a higher multiplicity static line at an interior point, the (quasi)-static line $P$ at a boundary point, and the static half-line $L_+$ with higher multiplicity at a boundary point. If any of these were to arise,  as in \cite[Section~7]{White_lectures}, we apply Proposition \ref{prop:sheeting} and choose scales for which the normalised graphical separation contradicts the following lemma.

A very similar technique was used in \cite[Theorems~9.1--9.2]{whi-size} to rule out higher multiplicity static planes, and adapted to the free boundary setting in \cite[Theorem~6.21]{FB-mean-convex}.

\begin{lemma}\label{lem:collapse}
	Let $I= \mathbb{R}$ or $I=[0, +\infty)$. There is no nonnegative smooth solution of $w_\tau=w_{xx}$ on $I\times(-2,0)$ satisfying
	\begin{equation}\label{eq:collapse-conditions}
		w(0,-1)=1,\qquad w(0,\tau)\le\sqrt{-\tau}
		\quad(-1\le\tau<0).
	\end{equation}
	and a homogeneous Neumann boundary condition at 0 when $I=[0, +\infty)$.
\end{lemma}
\begin{proof}
	Choose $r_0>0$ so that $w(x,-1)\ge1/2$ for $x\in I\cap[-r_0, r_0]$ and let
	\[
	\varphi(x,\tau)=\frac12\exp\!\left(-\frac{\pi^2(\tau+1)}{4r_0^2}\right)\cos\!\left(\frac{\pi x}{2r_0}\right), 
	\]
	which solves the heat equation, i.e. $\varphi_\tau=\varphi_{xx}$ on $\mathbb{R}$, and satisfies $\varphi_x(0, \tau)=0$.
	For every $b\in(-1,0)$, comparison on $(I\cap[-r_0,r_0])\times[-1,b]$ gives $w\geq\varphi$.
	Since $b<0$ is arbitrary,
	\[
	w(0,\tau)
	\ge\frac12\exp\!\left(-\frac{\pi^2}{4r_0^2}\right)>0
	\qquad(-1\le\tau<0),
	\]
	contradicting $w(0,\tau)\le\sqrt{-\tau}$ as $\tau\uparrow0$.
\end{proof}

\begin{proposition}\label{prop:static-exclusion}
Let $X_0=(p_0,\Tmax)\in\spt\Gamma$. If $p_0\in\Omega$, every tangent flow at $X_0$ has multiplicity one. If $p_0\in S$, a tangent flow of type \textup{(B2)} has multiplicity one, and type \textup{(B3)} does not occur.
\end{proposition}

\begin{proof}
Suppose that a tangent flow at $X_0$ is an interior line of multiplicity $q\ge2$, the line $P$ of multiplicity $q\ge1$, or the half-line $L_+$ of multiplicity $q\ge2$. At a boundary point the tangent flow is unique by Proposition~\ref{prop:tangent-flow-classification}. In the interior, all tangent flows have the same Gaussian density by \cite[Proposition~6.2 and Theorem~6.4]{edelen18}. Since the density of a static line is its multiplicity, every interior tangent line has multiplicity $q$, although its direction may depend on the sequence.

Put $\theta(t)=\sqrt{\Tmax-t}$. Any sequence of rescalings by $\theta(t)^{-1}$ as $t\uparrow\Tmax$ has a subsequence converging to a static flow with multiplicity $q$. Proposition~\ref{prop:sheeting} therefore gives the following descriptions for all sufficiently late times, and we define the separation function $d(t)$ as follows.
\begin{itemize}[leftmargin=*]
	\item \textit{The interior case.} On each of the two outermost sheets in $\overline B_{\theta(t)}(p_0)$, choose the point closest to $p_0$, and let $d(t)$ be the smaller of their distances to the opposite sheet. Notice that $d(t)=o(\theta(t))$ by Proposition \ref{prop:sheeting}.
	
	\item \textit{$\Sigma=P$.} Let $d(t)$ be the distance from $0$ to the first intersection of $\Gamma_t$ with $\{(x,0):0\le x\le\theta(t)\}$. In the fixed coordinates centred at $p_0$, Proposition~\ref{prop:sheeting} gives exactly $q$ transverse intersections. By convexity, all intersections have $x>0$ (see Remark \ref{rmk:convexity}).
	
	\item \textit{$\Sigma=L_+$.} Here $q=2$ by Proposition~\ref{prop:sheeting}. Let $d(t)$ be the arclength separation of the two endpoints along the short boundary arc about $p_0$.
\end{itemize}
By sheeting and embeddedness, $d(t)$ is positive and continuous, and by Proposition \ref{prop:sheeting}
\begin{equation}\label{eq:normalised-separation}
	\delta(t)\coloneqq\frac{d(t)}{\theta(t)}\to 0
	\quad\text{as }t\uparrow\Tmax.
\end{equation}
Choose $t_i\uparrow\Tmax$ such that
\begin{equation}\label{eq:tail-max}
	\sup_{t_i\le t<\Tmax}\delta(t)
	\le(1+i^{-1})\delta(t_i),
	\qquad \lambda_i=\theta(t_i)^{-1}.
\end{equation}
After passing to a subsequence, the rescaled flows \eqref{eq:rescale} converge to $q|\Sigma|$.  Proposition~\ref{prop:sheeting} gives ordered graphs $u_i^1<\cdots<u_i^q$ on domains exhausting $I\times(-2,0)$, where $I=\Rbb$ for lines and $I=[0,\infty)$ for the half-line. Also, $u_i^j\to0$ locally smoothly. Set
\[
h_i=
\begin{cases}
	u_i^1,&\Sigma=P,\\
	u_i^q-u_i^1,&\text{otherwise}.
\end{cases}
\]
The graphical equation for $u_i^1$ in the first case, and the difference of the equations for the outermost sheets in the other cases, give
\begin{equation}\label{eq:linear-collapse}
	(h_i)_\tau=A_i(h_i)_{zz}+B_i(h_i)_z+C_ih_i, 
\end{equation}
where $(A_i,B_i,C_i)\to(1,0,0)$ locally smoothly and $z=y$ in the interior and $\Sigma=P$ cases, and $z=x$ for the Fermi half-line graphs in the case $\Sigma=L_+$. The functions $h_i$ are positive (in the case $\Sigma=P$ this is a consequence of convexity as pointed out in Remark \ref{rmk:convexity}), and in the half-line case they satisfy homogeneous Neumann data at $0$.

We now relate $h_i$ to $d(t)$. Let $t=\Tmax+\lambda_i^{-2}\tau$. In both boundary cases, by definition, $h_i(0,\tau)=\lambda_i d(t)$.
In the interior case, let $s_i(\tau)$ be the $y$-coordinate of the closest point to the origin realising $d(t)$. Sheeting gives $s_i\to0$ uniformly on compact time intervals. Moreover, since the graphs have slopes $o(1)$ by Proposition \ref{prop:sheeting}, projecting the segment defining $\lambda_i d(t)$ onto the fixed normal gives
\begin{equation}\label{eq:relative-gap}
 h_i(s_i(\tau),\tau)=(1+o(1))\lambda_i d(t)
 =(1+o(1))\sqrt{-\tau}\,\delta(t).
\end{equation}
uniformly on compact time intervals in $(-2,0)$.

Normalise by setting $w_i=h_i/\delta(t_i)$. Equations \eqref{eq:tail-max} and \eqref{eq:relative-gap} yield
\[
w_i(s_i(-1),-1)\to1,\qquad
w_i(s_i(\tau),\tau)\le\sqrt{-\tau}+o(1)
\quad(-1\le\tau<0).
\]
Harnack's inequality and parabolic estimates \cite{Lieberman} give locally uniform bounds for $w_i$ and its spatial derivative, including at the Neumann boundary. Therefore, up to a subsequence, the $w_i$ converge locally smoothly to a non-negative solution $w$ of $w_\tau=w_{zz}$, with homogeneous Neumann data when $I=[0,\infty)$. Since $s_i\to0$,
\[
w(0,-1)=1,\qquad w(0,\tau)\le\sqrt{-\tau}
\quad(-1\le\tau<0).
\]
This contradicts Lemma~\ref{lem:collapse}.
\end{proof}

\begin{remark}[The role of convexity]
    Convexity is crucial to rule out the (quasi)-static line $P$ as a tangent flow. Indeed, notice that for general free boundary curve shortening flows in non-convex domains (see the discussion in \cite{FB-mean-convex} at p.~770, following Theorem~1.1 and at p.~774, in footnote~3) this possibility can in fact arise. In the proof of Proposition \ref{prop:static-exclusion}, convexity appears indirectly, as it guarantees positivity of the functions $h_i$ in the case $\Sigma=P$ (see also Remark \ref{rmk:convexity}), which is used to extract a non-negative solution of the heat equation to which Lemma \ref{lem:collapse} applies.
\end{remark}

\section{Proof of Theorem \ref{thm:main}}\label{sec:completion}

\begin{proposition}\label{prop:boundary-sing}
	At a finite-time boundary singular point, the unique tangent flow is the multiplicity-one shrinking semicircle.
\end{proposition}
\begin{proof}
	By Propositions~\ref{prop:tangent-flow-classification} and~\ref{prop:static-exclusion}, the unique tangent flow is either a multiplicity-one orthogonal half-line or a multiplicity-one shrinking semicircle. In the half-line case, the reflected Gaussian density is $1$, so Edelen's local regularity result \cite[Corollary~8.8]{edelen18} would make the point regular, a contradiction.
\end{proof}

\begin{proof}[Proof of Theorem~\ref{thm:main}]
	If $\Tmax=\infty$, apply Proposition~\ref{prop:immortal}. Suppose $\Tmax<\infty$. By \eqref{eq:curvature-blowup}, choose $t_i\uparrow\Tmax$ and $p_i\in\Gamma_{t_i}$ with $|\kappa(p_i,t_i)|\to\infty$. Compactness of $\oOm$ gives a subsequence with $p_i\to p_0$ and the point $(p_0,\Tmax)$ is singular. If $p_0\in\Omega$, Propositions~\ref{prop:tangent-flow-classification} and~\ref{prop:static-exclusion} imply that every tangent flow there is a multiplicity-one static line. The Gaussian density is therefore $1$, and White's local regularity theorem \cite[Theorems~3.1 and~3.5]{whi05localreg}, applied in a neighbourhood disjoint from $S$, would make the point regular. Thus $p_0\in S$, and Proposition~\ref{prop:boundary-sing} identifies the unique tangent flow as the multiplicity-one shrinking semicircle.
	
	Let $t_i\uparrow\Tmax$ now be arbitrary and put $\lambda_i=(\Tmax-t_i)^{-1/2}$. The rescaled flows \eqref{eq:rescale} about $(p_0,\Tmax)$ converge subsequentially to
	$\Sigma_\tau=\sqrt{-2\tau}\,\Sbb^1_+$
	with multiplicity one. Choose a good time $\tau_0\in(-2,-1)$. As in the compact case of the proof of Proposition~\ref{prop:tangent-flow-classification}, local support convergence and connectedness imply that the whole curve $\Gamma^i_{\tau_0}$ lies in a fixed ball for all large $i$. Thus Remark~\ref{rmk:good-times} gives, after a further subsequence and reparametrisation, global $C^{1,\alpha}$ convergence to $\Sigma_{\tau_0}$, for $0<\alpha<1/2$.
	
	The boundaries $S_i$ converge smoothly to $P$ near the compact limiting flow on $[\tau_0,-1/2]$. Therefore, by continuous dependence and positive-time smoothing for free boundary curve shortening flow, for every $0<\epsilon<-1/2-\tau_0$, $\Gamma^i_\tau \to \Sigma_\tau$ smoothly on $[\tau_0+\epsilon,-1/2]$. 
	In particular, we obtain smooth convergence of $\Gamma^i_{-1}$ to $\Sigma_{-1}$.
	Finally, dividing by $\sqrt2$ gives the convergence asserted in Theorem~\ref{thm:main}\textup{(ii)} along this subsequence.
	
	Since the sequence $t_i\uparrow\Tmax$ was arbitrary and the tangent flow is unique, the rescaled curves converge smoothly to the unit semicircle.
\end{proof}

\bibliographystyle{amsalpha}
\bibliography{references}
\end{document}